\documentclass[preprint,12pt]{elsarticle}

\usepackage{amssymb}
\usepackage{amsmath}

\usepackage[T2A]{fontenc} 

\newtheorem{theorem}{Theorem}
\newtheorem{proposition}{Proposition}
\newtheorem{corollary}{Corollary}

\journal{Indagationes Mathematicae}

\begin{document}

\begin{frontmatter}



\title{Power-weighted Hardy type integral inequalities with additional terms on finite intervals} 


\author{Ramil Nasibullin} 
\ead{NasibullinRamil@gmail.com}

\affiliation{organization={N.I. Lobachevsky Institute of Mathematics and Mechanics, Kazan Federal University},
            addressline={35, Kremlievskaia Str.},
            city={Kazan},
            postcode={420008},
            country={Russia}}
\affiliation{organization={Institute of Software Development and Engineering, Innopolis University},
            addressline={1, Universitetskaya Str.},
            city={Innopolis},
            postcode={420500},
            country={Russia}}

\begin{abstract}
In this paper we deal with integral Hardy type inequalities on finite segments. The interval is assumed to be finite and avoiding the origin. We prove new sharp $L_2$-inequalities and their $L_p$-analogues. Constants in the proven inequalities depend on the first root of the corresponding Lamb type equation for the Bessel function. In the $L_2$-case the extremal function is found. We consider Hardy-type inequalities in differential form. Using the one-dimensional Hardy inequalities, we establish an optimal multi-dimensional version of the power-weighted Hardy inequality in differential form on annuli.
\end{abstract}




\end{frontmatter}

\section{Introduction}

Applications of Hardy type inequalities in mathematical physics were a cause to their wide development. There are many papers and books devoted to different generalizations and modifications of Hardy's inequalities (see, for example, \cite{BEL}--\cite{Gadjev_arxiv} and references therein).   The original Hardy theorem from \cite{HLP} states that if $p\geq 1$ and $s>1$, then for a given function $f:(0,\infty)\to[0,\infty)$ obeying the condition $f/t^{s/p-1}\in L^p(0,\infty)$ the following inequality holds 
$$
\int\limits_0^{+\infty}  \left(\int\limits_0^tf(\tau)d\tau\right)^p \frac{dt}{t^{s}} \leq \left(\frac{p}{s-1}\right)^p\int\limits_0^{+\infty}\frac{
f^{p}(t)}{t^{s-p}}dt.
$$
The constant $\left(\frac{p}{s-1}\right)^p$ in this inequality is the best possible, that is,
it is maximal among all possible ones (see \cite{Avkh_23_Ufa} for more information).

As we said, Hardy inequalities have been generalized  and modified in numerous  ways, and the literature devoted to such inequalities is wide and different. For example,  Hardy type inequalities on finite intervals are known. The interval is assumed to be finite and avoiding the origin. Namely,  in \cite{GePa, DiGaIs, Wid58, Wid61}, the behaviour of the best possible constant in the general $L_p$-Hardy inequality 
\begin{equation}\label{DGI_Ineq}
    \int\limits_a^b\left(\int\limits_a^tf(\tau)d\tau\right)^p\frac{dt}{t^p}dt\leq d_p(a;b)\int\limits_a^b f^p(t)dt,
\end{equation}
 where  $a$ and $b$ are any fixed numbers with $0 < a < b < \infty$, $p > 1$, $f(t)\geq 0$ and $f^p$ is integrable over $[a,b]$, have been studied.  

 In \cite{DiGaIs}, D.K.~Dimitrov, I.~Gadjev and M.E.H.~Ismail proved  the sharp  $L_2$-inequality
\begin{equation}\label{DGI_Ineq}
    \int\limits_a^b \left(\int\limits_a^tf(\tau)d\tau\right)^2 \frac{dt}{t^2}\leq \frac{4}{1+4\lambda^2}\int\limits_a^b f^2(t)dt,
\end{equation}
where $\lambda$ is the only solution of the equation
$$
\tan \left(\lambda\log\frac{b}{a}\right)+2\lambda = 0, \quad\text{in the interval}\quad \lambda\in\left(\frac{\pi}{2\log\frac{b}{a}}, \frac{\pi}{\log\frac{b}{a}}\right).
$$
Moreover,  the authors found the extremal function 
$$
y_{a,b}(t) =t^{-1/2}\left(2\lambda \cos\left(\log\frac{t}{a}\right)+\sin\left(\lambda \log\frac{t}{a}\right)\right)
$$
at which in (\ref{DGI_Ineq}) the equality is attained. 

In \cite{GePa}, inequality (\ref{DGI_Ineq}) was extended and generalized by F. Gesztesy and M.M.H.~Pang. They established Hardy inequalities in integral and differential forms.   Using the one-dimensional Hardy-type results in differential form, they derive an optimal multi-dimensional version of the power-weighted Hardy inequality on annuli. One of the  optimal power-weighted Hardy type inequalities reads as follows:
\begin{equation}\label{GePa_Ineq}
    \int\limits_a^b \left(\int\limits_a^tf(\tau)d\tau\right)^2  \frac{dt}{t^{2-s}}\leq (4^{-1}(1-s)^2+\lambda^2)\int\limits_a^b f^2(t)t^s dt,
\end{equation}
where $s\in\mathbb{R}\setminus\{1\}$ and $\lambda$ is the unique solution  of
$$
\tan \left(\lambda \log\frac{b}{a}\right)+\frac{2\lambda}{1-s} = 0, \quad\text{in the interval}\quad \lambda\in\left(0, \frac{\pi}{\log\frac{b}{a}}\right)\setminus \frac{\pi}{2\log\frac{b}{a}}.
$$
In addition, F. Gesztesy and M.M.H.~Pang found the extremal function at which in the inequality equality is attained. It is easy to see that if $s = 0$, then inequality (\ref{GePa_Ineq}) implies inequality (\ref{DGI_Ineq}).

Also in \cite{GePa} the following sharp inequality in differential form
\begin{equation}\label{in_Ge_Pa}
    \left(\frac{(1-\alpha)^2}{4}+\frac{\pi^2}{\log^2\frac{b}{a}}\right)\int\limits_a^b \frac{|y(t)|^2}{t^{2-\alpha}}dt\leq \int\limits_a^b |y'(t)|^2t^\alpha dt
\end{equation}
are proved for any $y\in W_0^{1,2}((a,b))$. Here by $W_0^{1,2}((a,b))$ the completion of $C^\infty_0((a,b))$ with respect to the corresponding Sobolev norm is denoted. This means that $y$ satisfies two boundary conditions $y(a) = y(b) = 0$. Notice that in  (\ref{DGI_Ineq}) and (\ref{GePa_Ineq}) there is only one zero boundary condition.

We generalize and modify the above inequality. Namely, we prove that if  $q>0$, $s>0$, $\nu\in[0,\frac{1}{q}]$, $a$ and $b$ is any fixed numbers with $0 < a < b < \infty$, $y:[a,b]\to \mathbb{R}$ is  an absolutely continuous function such that $y(a) =y(b)= 0$ and $y'^2/t^{s-1}\in L^1[a,b]$,  then
\begin{multline}\label{th3_ineq_intr}
    \frac{s^2}{4}\int\limits_a^b\frac{y^2(t)}{t^{s+1}}dt+\frac{q^2j_\nu^2}{4\log^2\frac{b}{a}} \int\limits_a^b\frac{y^2(t)}{t^{s+1} \log^{2-q}\frac{t}{a}}dt+\\ \frac{1-\nu^2q^2}{4}\int\limits_a^b\frac{y^2(t)}{t^{s+1}\log^2\frac{t}{a}}dt\leq\int\limits_a^b \frac{y'^2(t)}{t^{s-1}}dt,
\end{multline}
where  $j_\nu$  is the first zero of the Bessel function  
$$
J_\nu(t) = \sum_{k=0}^\infty\frac{(-1)^kt^{2k+\nu}}{2^{2k+\nu}k!\Gamma(k+1+\nu)}, \quad t\in[0,1],
$$
of order $\nu\geq 0$ (see \cite{Watson} for more information). Furthermore, if $\nu\in(0,\frac{1}{q}]$ then the equality in this inequality is attained for 
$$
Y_0(t) = t^{\frac{s}{2}}\sqrt{\log\frac{t}{a}} J_\nu \left(j_\nu\left(\frac{\log \frac{t}{a}}{\log\frac{b}{a}}\right)^{q/2} \right).
$$

As a corollary of our result in the case $\nu  = 1/2$ and $q=2$, using known facts
$$
J_{1/2}(x) = \sqrt{\frac{2}{\pi}}\frac{\sin x}{\sqrt{x}} \quad \text{and}\quad j_{1/2} = \pi,
$$
we get inequality (\ref{in_Ge_Pa}). In addition, we obtain a multidimensional analogue of (\ref{th3_ineq_intr}) on annuli. 

It is interesting to compare the inequalities (\ref{in_Ge_Pa}) and (\ref{th3_ineq_intr}) with the result of the paper \cite{Avk_18} by F.G. Avkhadiev. In \cite{Avk_18}, the author proved that the inequality 
$$
\int\limits_q^1 \frac{y^2(t)}{\rho^2(t)}tdt+\frac{\pi^2}{4\log^2 q}\int\limits_q^1 \frac{y^2(t)}{t}dt< \int\limits_q^1 y'^2(t)tdt,
$$
where
$$
\rho(t) = \frac{2t\log q}{\pi} \sin\frac{\pi\log t}{\log q},
$$
is valid for an absolutely continuous function $y$ with $y(q) = y(1) = 0$. We see that the inequality are proved for functions with two zero boundary conditions. Moreover, using this inequality, F.G.~Avkhadiev obtained inequalities on annuli and in doubly-connected domains of finite modulus.  

Also, $L_p$-Hardy inequalities on finite intervals are known. I.~Gadjev \cite{Gadjev_26} (see also \cite{Gadjev_arxiv}) studies the behaviour of the smallest possible constants $d_{p}(a, b)$ and gets the exact rate of convergence of $d_{p}(a, b)$.  Furthermore, the “almost extremal” function is found. He obtained two side bilateral estimates of $d_{p}(a, b)$. In fact, I.~Gadjev proved that if $p\geq 2$ and $0 < a< b < +\infty$, then there exist positive constants $c_1 = c_1(p)$ and $c_2 = c_2(p)$, depending only on $p$, such that 
\begin{equation}\label{billat_ineq}
\left(\frac{p}{p-1}\right)^p\left(1-\frac{c_1}{\log^2\frac{b}{a}}\right)\leq d(a,b)\leq \left(\frac{p}{p-1}\right)^p\left(1+\frac{c_2}{\log^2\frac{b}{a}}\right)^{-1}.
\end{equation}

In the present paper, as a corollary of our main results, we get the following $L_p$-inequality 
$$
\int\limits_a^b \frac{|y(t)|^p}{t^{p}}dt \leq \left(\frac{p}{p-1}\right)^p\left(1+\frac{2p\lambda^2}{(p-1)^2}\right)^{-1}\int\limits_a^b\frac{|y'(t)|^{p}}{t^{p-2}}dt,
$$
where $\lambda$ is the first root of the corresponding Lamb type equation. If $p=2$, then we have inequality (\ref{DGI_Ineq}).

D.K.~Dimitrov, I.~Gadjev and M.E.H.~Ismail \cite{DiGaIs} in an $L_2$-case have proven truly beautiful and almost complete  inequality, in which an exact constant and an extremal function are found.  This result is the direct generalization of another beautiful sharp result  by V.I.~Levin~\cite{Levin}
\begin{equation}\label{Levin_Ineq}
    \int\limits_0^1\frac{y^2(t)}{t^2}dt < 4\int\limits_0^1 y'^2(t)dt,
\end{equation}
where $y:[0,1]\to \mathbb{R}$ is a continuously differentiable function such that $y(0) = 0$, $y\not\equiv 0$ and $y'\in L^2[0,1]$.  The constant $4$ is sharp although there is no function for which in the inequality equality is actually attained.

Inequality (\ref{Levin_Ineq}) is a differential form of (\ref{DGI_Ineq}). In \cite{AW_Zamm}, F.G.~Avkhadiev and \mbox{K.-J.~Wirths}  established analogues of (\ref{Levin_Ineq})  with an additional positive term. Namely, they proved that if $q>0$ and $\nu \in \left[0,\frac{1}{q}\right]$, $y\in C^1([0,1])$ such that  $y(0)=0$, then the sharp inequality holds
\begin{equation}\label{avkh_wir1}
(1-\nu^2q^2)\int\limits_0^1\frac{y^2(t)}{t^2} d t+q^2\lambda^2\int\limits_0^1\frac{y^2(t)}{t^{2-q}} d t \leq  4\int\limits_0^1  y'^2(t) dt.
\end{equation}
Here $\lambda$ is a constant defined as  the first positive root of the following Lamb equation 
$$
J_\nu(\lambda)+q\lambda J_\nu'(\lambda)=0, \quad \lambda\in (0,j_\nu).
$$

The goal of the present paper is to generalize inequality (\ref{DGI_Ineq}) by additional positive terms. For example, we get the following theorem.
\begin{theorem}\label{th1}
    Let $\nu\in[0,\frac{1}{2}]$, $a$ and $b$ be fixed numbers with $0 < a < b < \infty$. Suppose $y:[a,b]\to \mathbb{R}$ is an absolutely continuous function  such that $y(a) = 0$ and $y'\in L^2[a,b]$. Then
    $$
\frac{1+4\lambda^2}{4}\int\limits_a^b\frac{y^2(t)}{t^2}dt+ \frac{1-4\nu^2}{4}
\int\limits_a^b\frac{y^2(t)}{t^2\log^2\frac{t}{a}}dt\leq\int\limits_a^b y'^2(t)dt,
    $$
where $\lambda$ is the first root of the following Lamb type equation
$$
\frac{\log\frac{eb}{a}}{2\log\frac{b}{a}}J_\nu \left(\lambda\log \frac{b}{a} \right)+\lambda J'_\nu \left(\lambda\log \frac{b}{a} \right)=0.
$$
Moreover, if $\nu\in (0,1/4]$ then the equality in this inequality is attained for 
$$
y_0(t) = \sqrt{t \log\frac{t}{a}} J_\nu \left(\lambda\log \frac{t}{a} \right), \quad \lambda\in \left(0,\frac{j_\nu}{\log \frac{b}{a}}\right).
$$
\end{theorem}

Note that since the equality in (\ref{GePa_Ineq}) is  attained, it is unusual, and even it looks strange, to set a question on adding positive terms (see \cite{
Nas_Rew, BM, HoHoL, Tid, AW_Zamm, AW_Lamb, Nas_IzvM, Nas_Sib, FMT}). Here, we obtain new sharp $L_2$-inequalities and $L_p$-inequalities. Moreover, in the $L_2$-case, the extremal function is found. We consider inequalities in differential form. Our results imply the above mentioned inequalities. We should add that we use standard methods for special functions.  The main achievement is the construction of these functions and the idea  of adding additional terms.

\section{$L_2$-inequalities for functions with a zero boundary condition}

In this section, we will consider Hardy type inequalities for functions with one zero boundary condition. Besides, we provide the proof of Theorem \ref{th1}. In the sequel, we need some properties of the function $y_0(t)$. Using the well known equality for the Bessel functions
$$
J'_\nu(z) = \frac{1}{2}\left(J_{\nu-1}(z)-J_\nu (z)\right),
$$
direct computations give 
$$
y_0'(t) = \frac{1}{2\sqrt t}\left(\frac{\log\frac{et}{a}}{\sqrt{\log\frac{t}{a}}}J_\nu \left(\lambda\log \frac{t}{a} \right)+ 2\lambda \sqrt{\log\frac{t}{a}}  J'_\nu \left(\lambda\log \frac{t}{a} \right)\right),
$$
$$
\frac{y'_0(t)}{y_0(t) } = \frac{1}{t}\left(\frac{\log\frac{et}{a}}{2\log\frac{t}{a}}+\lambda\frac{J'_\nu \left(\lambda\log \frac{t}{a} \right)}{J_\nu \left(\lambda\log \frac{t}{a} \right)}\right)= 
$$
$$
\frac{1}{t}\left(\frac{\log\frac{et}{a}}{2\log\frac{t}{a}}+\frac{\lambda}{2}\left(\frac{J_{\nu-1} \left(\lambda\log \frac{t}{a} \right)}{J_\nu \left(\lambda\log \frac{t}{a} \right)}-\frac{J_{\nu+1} \left(\lambda\log \frac{t}{a} \right)}{J_\nu \left(\lambda\log \frac{t}{a} \right)}\right)\right)
$$
and
\begin{equation}\label{f_difeq_th1}
\frac{y''_0(t)}{y_0(t)} = -\frac{1+4\lambda^2}{4t^2}-\frac{1-4\nu^2}{4t^2\log^2\frac{t}{a}}.
\end{equation}
Applying integrating by parts, we have
$$
0\leq \int\limits_a^b\left(y'(t)-\frac{y_0'(t)}{y_0(t)}y(t)\right)^2dt=
$$
$$
= \int\limits_a^by'^2(t)dt-\int\limits_a^b  \frac{y_0'(t)}{y_0(t)} d y^2(t)+\int\limits_a^b  \left(\frac{y_n'(t)}{y_n(t)} \right)^2 y^2(t)dt=
$$
$$
\int\limits_a^by'^2(t)dt+\int\limits_a^b\left(\frac{y_0'^2(t)}{y_0^2(t)}+\left(\frac{y_0'(t)}{y_0(t)}\right)^{'}\right)y^2(t)dt
+\lim_{t\to a+}\frac{y'_0(t)}{y_0(t)}y^2(t)-\lim_{t\to b-}\frac{y'_0(t)}{y_0(t)}y^2(t).
$$
Consequently, 
$$
\int\limits_a^by'^2(t)dt+\int\limits_a^b\frac{y_0''(t)}{y_0(t)}y^2(t)dt\geq
\lim_{t\to b-}\frac{y'_0(t)}{y_0(t)}y^2(t)-\lim_{t\to a+}\frac{y'_0(t)}{y_0(t)}y^2(t).
$$
As a consequence of the condition $y'\in L^2[a,b]$ via the Cauchy-Schwartz inequality 
$$
y^2(t)\leq \left(\int\limits_a^t|y'(\tau)|d\tau\right)^2\leq (t-a)\int\limits_a^t|y'(\tau)|^2d\tau,
$$
we obtain $y^2(t)/(t-a)\to 0$ as $t\to a+$. 

Taking into account  the following asymptotic forms
$$
J_\nu(z) \sim \frac{1}{\Gamma(\nu+1)}\left(\frac{z}{2}\right)^\nu
$$
for  small arguments $z$, we obtain
$$
y'_0(t)\sim \frac{1}{\left(\log\frac{t}{a}\right)^{\frac{1}{2}-\nu}}, \quad t\to a, 
$$
$$
\frac{y'_0(t)}{y_0(t) } \sim \frac{1}{t}\left(\frac{\log\frac{et}{a}}{2\log\frac{t}{a}}+\frac{\nu}{\log\frac{t}{a}}\right), \quad t\to a .
$$
Hence
$$
\lim_{t\to a}(t-a)\frac{y'_0(t)}{y_0(t) } = \frac{1}{2}+\nu\quad \text{and} \quad \lim_{t\to a+}\frac{y'_0(t)}{y_0(t)}y^2(t)=0,
$$
$$
y'_0\in L^2[a,b] \quad \text{if} \quad \nu\in(0,1/2], \quad \text{and} \quad y'_0\notin L^2[a,b] \quad \text{if} \quad \nu =0.
$$
Therefore
$$
\int\limits_a^by'^2(t)dt+\int\limits_a^b\frac{y_0''(t)}{y_0(t)}y^2(t)dt\geq 0.
$$
The application of  equality (\ref{f_difeq_th1}) yields Theorem \ref{th1}.

In the next proposition, we will show that the constant is sharp in the case $\nu = 0$.

\begin{proposition}\label{prop1} For any $\varepsilon_0 >0$ there exists a function $g\in C^1[a,b]$ such that $g(a) = 0$ and $g'\in L^2[a,b] $ and 
$$
\frac{1+4\lambda^2}{4}\int\limits_a^b\frac{g^2(t)}{t^2}dt+ \left(\frac{1}{4}+\varepsilon_0\right)
\int\limits_a^b\frac{g^2(t)}{t^2\log^2\frac{t}{a}}dt > \int\limits_a^b g'^2(t)dt.
$$
\end{proposition}
\textbf{\textit{Proof of Proposition \ref{prop1}.}}  In the sequel we will use the function
$$
g_\varepsilon(t) = \sqrt{t}\left(\log\frac{t}{a}\right)^\frac{\varepsilon+1}{2} J_0\left(\lambda\log \frac{t}{a} \right).
$$
Consider the following difference
$$
A = 
 \frac{1+4\lambda^2}{4}\int\limits_a^b\frac{g_\varepsilon^2(t)}{t^2}dt+ \left(\frac{1}{4}+\varepsilon_0\right)
\int\limits_a^b\frac{g_\varepsilon^2(t)}{t^2\log^2\frac{t}{a}}dt- \int\limits_a^b g_\varepsilon'^2(t)dt .
$$
The computations in the proof of Theorem \ref{th1} give 
$$
A = \varepsilon\int\limits_a^b\frac{g_\varepsilon^2(t)}{t^2\log^2\frac{t}{a}}dt -\int\limits_a^b\left(g_\varepsilon'(t)-\frac{y_0'(t)}{y_0(t)}g_\varepsilon(t)\right)^2dt.
$$
One can show that 
$$
A = \varepsilon_0\int\limits_a^b\frac{J_0^2\left(\lambda\log \frac{t}{a} \right)}{t\left(\log \frac{t}{a} \right)^{1-\varepsilon}}dt -\frac{\varepsilon^2}{4} \int\limits_a^b\frac{J_0^2\left(\lambda\log \frac{t}{a} \right)}{t\left(\log \frac{t}{a} \right)^{1-\varepsilon}}dt.
$$
Obviously, $A> 0$ for sufficiently small $\varepsilon$.

Taking $\nu = 1/2$ in Theorem \ref{th1} and using the  well known identities for the Bessel function   
$$
J_{1/2}(x) = \sqrt\frac{2}{\pi}\frac{\sin x}{\sqrt x} \quad \text{and }\quad j_{1/2}=\pi,
$$
we get the following corollary.

\begin{corollary}
Let $a$ and $b$ be any fixed numbers with $0 < a < b < \infty$. Suppose $y:[a,b]\to \mathbb{R}$ is an absolutely continuous function such that $y(a) = 0$ and $y'\in L^2[a,b]$. Then
$$
\frac{1+4\lambda^2}{4}\int\limits_a^b\frac{y^2(t)}{t^2}dt\leq\int\limits_a^b y'^2(t)dt,
$$
where $\lambda$ is the first root of the following Lamb type equation
$$
\sin\left(\lambda \log\frac{b}{a}\right)+2\lambda \cos\left(\lambda\log \frac{b}{a} \right)=0.
$$
Moreover, the equality in this inequality is attained for 
$$
y_0(t) = \sqrt{t}\sin \left(\lambda\log \frac{t}{a} \right), \quad \lambda\in \left(0,\frac{\pi}{\log \frac{b}{a}}\right).
$$
\end{corollary}
It is easy to show that
$$
y'_0(t) = t^{-1/2}\left(2 \lambda \cos\left(\lambda\log\left(\frac{t}{a}\right)\right) + \sin\left(\lambda \log\left(\frac{t}{a}\right)\right)\right).
$$
Therefore, if $\nu  =1/2$, then Theorem \ref{th1} implies inequality (\ref{DGI_Ineq}).

\section{Power-weighted inequalities}

In this section, we prove power-weighted inequalities on finite intervals.  The inequalities are generalization of the results from the previous section. The main result reads as:

\begin{theorem}\label{th2}
Let $q>0$, $s>0$, $\nu\in[0,\frac{1}{q}]$, $a$ and $b$ be any fixed numbers with $0 < a < b < \infty$. Suppose $y:[a,b]\to \mathbb{R}$ is  an absolutely continuous function such that $y(a) = 0$ and $y'\in L^2[a,b]$. Then
$$
    \frac{s^2}{4}\int\limits_a^b\frac{y^2(t)}{t^{s+1}}dt+\frac{q^2\lambda^2}{4} \int\limits_a^b\frac{y^2(t)}{t^{s+1} \log^{2-q}\frac{t}{a}}dt+\frac{1-\nu^2q^2}{4}\int\limits_a^b\frac{y^2(t)}{t^{s+1}\log^2\frac{t}{a}}dt\leq\int\limits_a^b \frac{y'^2(t)}{t^{s-1}}dt,
    $$
where $\lambda$ is the first root of the following Lamb type equation
\begin{equation}\label{th2_eq}
    \left(s+\frac{1}{\log\frac{b}{a}}\right)J_\nu \left(\lambda\left(\log \frac{b}{a}\right)^{\frac{q}{2}} \right)+q\lambda \left(\log \frac{b}{a} \right)^{\frac{q}{2}-1}J'_{\nu} \left(\lambda\left(\log \frac{b}{a}\right)^{\frac{q}{2}} \right)=0.
\end{equation}
In addition, if $\nu\in(0,\frac{1}{q}]$ then the equality in this inequality is attained for 
$$
Y_0(t) = t^{\frac{s}{2}}\sqrt{\log\frac{t}{a}} J_\nu \left(\lambda\left(\log \frac{t}{a}\right)^{q/2} \right), \quad \lambda\log \left(\log \frac{b}{a}\right)^{q/2} \in (0,j_\nu ).
$$
\end{theorem}
\textbf{\textit{Proof of Theorem \ref{th2}}}. In the sequel we need some properties of the function $Y_0(t)$. Using the well known equality for the Bessel functions
$$
J'_\nu(z) = \frac{1}{2}\left(J_{\nu-1}(z)-J_\nu (z)\right),
$$
direct computations give 
$$
\frac{Y'_0(t)}{Y_0(t) } = \frac{1}{2t}\left(s+\frac{1}{\log\frac{t}{a}}+q\lambda \left(\log \frac{t}{a} \right)^{q/2-1}\frac{J'_{\nu} \left(\lambda\left(\log \frac{t}{a}\right)^{q/2} \right)}{J_\nu \left(\lambda\left(\log \frac{t}{a}\right)^{q/2} \right)}\right)=
$$
$$
\frac{1}{2t}\left(s+\frac{1}{\log\frac{t}{a}}+q\lambda \left(\log \frac{t}{a} \right)^{q/2-1}\left(\frac{J_{\nu-1} \left(\lambda\left(\log \frac{t}{a}\right)^{q/2} \right)}{J_\nu \left(\lambda\left(\log \frac{t}{a}\right)^{q/2} \right)}+\frac{J_{\nu+1} \left(\lambda\left(\log \frac{t}{a}\right)^{q/2} \right)}{J_\nu \left(\lambda\left(\log \frac{t}{a}\right)^{q/2} \right)}\right)\right)
$$

and
\begin{equation}\label{f_difeq}
\frac{t^2Y''_0(t)}{Y_0(t)} +(1-s)\frac{tY'_0(t)}{Y_0(t)}= -\frac{s^2}{4}-\frac{q^2\lambda^2}{4\log^{2-q}\frac{t}{a}}-\frac{1-\nu^2q^2}{4\log^2\frac{t}{a}}.
\end{equation}

We can prove that
$$
Y'_0/t^{(s-1)/2}\in L^2[a,b] \quad \text{if} \quad \nu\in(0,1/q], \quad \text{and} \quad Y'_0/t^{(s-1)/2} \notin L^2[a,b] \quad \text{if} \quad \nu =0.
$$
Using integrating by parts, we have
$$
0\leq \int\limits_a^b\frac{1}{t^{s-1}}\left(y'(t)-\frac{Y_0'(t)}{Y_0(t)}y(t)\right)^2dt=
$$
$$
= \int\limits_a^b\frac{y'^2(t)}{t^{s-1}}dt-\int\limits_a^b  \frac{1}{t^{s-1}}\frac{Y_0'(t)}{Y_0(t)} d y^2(t)+\int\limits_a^b  \frac{ y^2(t)}{t^{s-1}}\left(\frac{Y_0'(t)}{Y_0(t)} \right)^2dt=
$$
\begin{multline*}
\int\limits_a^b\frac{y'^2(t)}{t^{s-1}}dt+\int\limits_a^b\left(\frac{t^2Y_0''
(t)+(1-s)tY_0'(t)}{Y_0(t)}\right)\frac{y^2(t)}{t^{s+1}}dt
+\\ \lim_{t\to a+}\frac{y^2(t)}{t^{s-1}}\frac{Y'_0(t)}{Y_0(t)}-\lim_{t\to b-}\frac{y^2(t)}{t^{s-1}}\frac{Y'_0(t)}{Y_0(t)}.
\end{multline*}
Consequently, 
\begin{multline*}
\int\limits_a^by'^2(t)dt+\int\limits_a^b\left(\frac{t^2Y_0''
(t)+(1-s)tY_0'(t)}{Y_0(t)}\right)\frac{y^2(t)}{t^{s+1}}dt\geq
\\ \lim_{t\to b-}\frac{Y'_0(t)}{Y_0(t)}y^2(t)-\lim_{t\to a+}\frac{Y'_0(t)}{y_0(t)}y^2(t).
\end{multline*}
As a consequence of the condition $y'/t^{\frac{s-1}{2}}\in L^2[a,b]$ via the Cauchy-Schwartz inequality 
$$
y^2(t)\leq \left(\int\limits_a^t|y'(\tau)|d\tau\right)^2\leq \frac{t^{s}-a^s}{s}\int\limits_a^t\frac{|y'(\tau)|^2}{t^{s-1}}d\tau,
$$
we obtain $y^2(t)/(t^s-a^s)\to 0$ as $t\to a+$. 

Taking into account  the following asymptotic forms
$$
J_\nu(z) \sim \frac{1}{\Gamma(\nu+1)}\left(\frac{z}{2}\right)^\nu
$$
for  small arguments $z$, we obtain
$$
\frac{y'_0(t)}{y_0(t) } \sim \frac{1}{2t}\left(s+\frac{1}{\log\frac{t}{a}}+   \frac{\nu q }{\log \frac{t}{a}}\right).
$$
Hence
$$
\lim_{t\to a}(t^s-a^s)\frac{y'_0(t)}{y_0(t) } = \frac{sa^{s-1}}{2}\left(1+\nu q\right)\quad \text{and} \quad \lim_{t\to a+}\frac{y'_0(t)}{y_0(t)}y^2(t)=0.
$$
Therefore
$$
\int\limits_a^by'^2(t)dt+\int\limits_a^b\left(\frac{t^2Y''_0(t)}{Y_0(t)} +(1-s)\frac{tY'_0(t)}{Y_0(t)} \right)\frac{y^2(t)}{t^{s+1}}dt\geq 0.
$$
and
$$
\int\limits_a^by'^2(t)dt+\int\limits_a^b\left(\frac{t^2Y''_0(t)}{Y_0(t)} +(1-s)\frac{tY'_0(t)}{Y_0(t)} \right)\frac{y^2(t)}{t^{s+1}}dt\geq 0.
$$
To conclude the proof, it remains to use equality (\ref{f_difeq}). This completes the proof of Theorem \ref{th2}.

In the next proposition we will show that the constant is sharp in the case $\nu = 0$.

\begin{proposition}\label{prop2}For any $\varepsilon_0 >0$ there exists a function $g\in C^1[a,b]$ such that $g(a) = 0$ and $g'\in L^2[a,b] $ and 
  $$
    \frac{s^2}{4}\int\limits_a^b\frac{g^2(t)}{t^{s+1}}dt+\frac{q^2\lambda^2}{4} \int\limits_a^b\frac{g^2(t)}{t^{s+1} \log^{2-q}\frac{t}{a}}dt+\left(\frac{1}{4}+\varepsilon_0\right)\int\limits_a^b\frac{g^2(t)}{t^{s+1}\log^2\frac{t}{a}}dt >\int\limits_a^b \frac{g'^2(t)}{t^{s-1}}dt,
    $$
\end{proposition}
\textit{\textbf{Proof of Proposition \ref{prop2}.}} Consider the function  $g_\varepsilon$ defined by
$$
g_\varepsilon(t) = \sqrt{t^s}\left(\log\frac{t}{a}\right)^\frac{\varepsilon+1}{2} J_0\left(\lambda\log \frac{t}{a} \right).
$$
and the following difference
\begin{multline*}
A = \frac{s^2}{4}\int\limits_a^b\frac{g_\varepsilon ^2(t)}{t^{s+1}}dt+\frac{q^2\lambda^2}{4} \int\limits_a^b\frac{g_\varepsilon^2(t)}{t^{s+1} \log^{2-q}\frac{t}{a}}dt+\\ \left(\frac{1}{4}+\varepsilon_0\right)\int\limits_a^b\frac{g_\varepsilon^2(t)}{t^{s+1}\log^2\frac{t}{a}}dt -\int\limits_a^b \frac{g_\varepsilon'^2(t)}{t^{s-1}}dt.
\end{multline*}
Computations in the proof of Theorem \ref{th2} give that
$$
A = \varepsilon_0\int\limits_a^b\frac{g_\varepsilon^2(t)}{t^{s+1}\log^2\frac{t}{a}}dt -\int\limits_a^b\frac{1}{t^{s-1}}\left(g_\varepsilon'(t)-\frac{Y_0'(t)}{Y_0(t)}g_\varepsilon(t)\right)^2dt.
$$
One can show that 
$$
A = \varepsilon_0\int\limits_a^b\frac{J_0^2\left(\lambda\log \frac{t}{a} \right)}{t\left(\log \frac{t}{a} \right)^{1-\varepsilon}}dt -\frac{\varepsilon^2}{4} \int\limits_a^b\frac{J_0^2\left(\lambda\log \frac{t}{a} \right)}{t\left(\log \frac{t}{a} \right)^{1-\varepsilon}}dt.
$$
Obviously, $A> 0$ for sufficiently small $\varepsilon$.

The choice $\nu =1/2$ and $q=2$ in Theorem \ref{th2}, yields the following corollary.  
\begin{corollary}
    Let $q>0$, $s>0$, $a$ and $b$ be any fixed numbers with $0 < a < b < \infty$. Suppose $y:[a,b]\to \mathbb{R}$ is  an absolutely continuous function such that $y(a) = 0$ and $y'\in L^2[a,b]$. Then
    $$
   \left( \frac{s^2}{4}+\lambda^2\right)\int\limits_a^b\frac{y^2(t)}{t^{s+1}}dt\leq\int\limits_a^b \frac{y'^2(t)}{t^{s-1}}dt,
    $$
where $\lambda$ is the first root of the following Lamb type equation
$$
2\lambda \cos\left(\lambda\log \frac{b}{a} \right)+s\sin\left(\lambda \log\frac{b}{a}\right)=0.
$$
Moreover,  then the equality in this inequality is attained for 
$$
Y_0(t) = t^{\frac{s}{2}}\sin \left(\lambda  \log \frac{t}{a} \right), \quad \lambda\log \frac{b}{a}\in (0,\pi ).
$$
\end{corollary}

As a limit case as $q\to 0$, Theorem \ref{th2} presents the following statement. 

\begin{corollary}\label{cor3}
    Let $s>0$, $a$ and $b$ be any fixed numbers with $0 < a < b < \infty$. Suppose $y:[a,b]\to \mathbb{R}$ is  an absolutely continuous function such that $y(a) = 0$ and $y'\in L^2[a,b]$. Then
   $$
    \frac{s^2}{4}\int\limits_a^b\frac{y^2(t)}{t^{s+1}}dt+\frac{1}{4} \int\limits_a^b\frac{y^2(t)}{t^{s+1} \log^{2}\frac{t}{a}}dt\leq\int\limits_a^b \frac{y'^2(t)}{t^{s-1}}dt.
$$
\end{corollary}
\textbf{\textit{Proof of Corollary \ref{cor3}.}} If $q = \frac{s\log\frac{b}{a}+1}{\nu}$, then taking into account the identity for the Bessel's functions 
$$
\nu J_\nu(z)+zJ'_\nu(z) = z J_{\nu-1}(z), 
$$
equation  (\ref{th2_eq}) has the following form 
$$
\lambda\left(\log \frac{b}{a}\right)^{q/2} J_{\nu-1}\left(\lambda\left(\log \frac{b}{a}\right)^{q/2}\right) = 0  
$$
Hence
$$
\quad \quad \lambda\left(\log \frac{b}{a}\right)^{q/2} = j_{\frac{s\log\frac{b}{a}+1}{q}-1}.
$$
It is known (see, for instance, \cite{AW_Lamb} ) that 
$$
\lim_{q\to 0} qj_{p/q-1}= p.
$$
Thus
$$
q\lambda  = s\log\frac{b}{a}+1 \quad\text{as} \quad q\to 0.
$$
Therefore
\begin{multline*}
        \frac{s^2}{4}\int\limits_a^b\frac{y^2(t)}{t^{s+1}}dt+\frac{(s\log\frac{b}{a}+1)^2}{4} \int\limits_a^b\frac{y^2(t)}{t^{s+1} \log^{2}\frac{t}{a}}dt+\\
        \frac{1-(s\log\frac{b}{a}+1)^2}{4}\int\limits_a^b\frac{y^2(t)}{t^{s+1}\log^2\frac{t}{a}}dt\leq\int\limits_a^b \frac{y'^2(t)}{t^{s-1}}dt.
\end{multline*}
Using the last inequality, we complete the proof of Corollary \ref{cor3}. 

\section{$L_2$-inequalities for functions with two zero boundary conditions}
In the previous section we considered continuously differentiable functions such that $y(a) = 0$ and we  used
$$
\lambda\in \left(0,\frac{j_\nu}{\left(\log \frac{b}{a}\right)^{q/2} }\right).
$$
If we consider continuously differentiable functions such that $y(a) = 0$ and $y(b) = 0$ then we  will be able to include the limiting case 
$$
\lambda = \frac{j_{n}}{\left(\log\frac{b}{a}\right)^{q/2}}.
$$
To achieve the goal we follow the proof of the previous theorem and in addition we will have
$$
y^2(b)\leq \left(\int\limits_{t}^b|y'(\tau)|d\tau\right)^2\leq \frac{b^{s}-t^s}{s}\int\limits_{t}^b\frac{|y'(\tau)|^2}{\tau^{s-1}}d\tau,
$$
and
\begin{multline*}
0\leq \lim_{t\to b-}\frac{y^2(t)}{t^{s-1}}\frac{Y'_0(t)}{Y_0(t)} = \\ \lim_{t\to b-} \frac{y^2(t)}{2t^s}\left(s+\frac{1}{\log\frac{t}{a}}+q\lambda \left(\log \frac{t}{a} \right)^{\frac{q}{2}-1}\frac{J'_{\nu} \left(\lambda\left(\log \frac{t}{a}\right)^{\frac{q}{2}} \right)}{J_\nu \left(\lambda\left(\log \frac{t}{a}\right)^{\frac{q}{2}} \right)}\right)=
\\
\frac{q J'_{\nu} \left(j_\nu \right) }{s\log \frac{b}{a}}\lim_{t\to b-} \frac{1-\frac{t^s}{b^{s}}}{J_\nu \left(j_\nu\left(\frac{\log \frac{t}{a}}{\log\frac{b}{a}}\right)^{q/2} \right)} \int\limits_{t}^b\frac{|y'(\tau)|^2}{\tau^{s-1}}d\tau =0.
\end{multline*}
Therefore the following theorem is valid.

\begin{theorem}\label{th3}    Let $q>0$, $s>0$, $\nu\in[0,\frac{1}{q}]$, $a$ and $b$ be any fixed numbers with $0 < a < b < \infty$. Suppose $y:[a,b]\to \mathbb{R}$ is  an absolutely continuous function such that $y(a) =y(b)= 0$ and $y'^2/t^{s-1}\in L^1[a,b]$. Then
\begin{multline}\label{th3_ineq}
    \frac{s^2}{4}\int\limits_a^b\frac{y^2(t)}{t^{s+1}}dt+\frac{q^2j_\nu^2}{4\log^2\frac{b}{a}} \int\limits_a^b\frac{y^2(t)}{t^{s+1} \log^{2-q}\frac{t}{a}}dt+\\ \frac{1-\nu^2q^2}{4}\int\limits_a^b\frac{y^2(t)}{t^{s+1}\log^2\frac{t}{a}}dt\leq\int\limits_a^b \frac{y'^2(t)}{t^{s-1}}dt.
\end{multline}
Moreover, if $\nu\in(0,\frac{1}{q}]$ then the equality in this inequality is attained for 
$$
Y_0(t) = t^{\frac{s}{2}}\sqrt{\log\frac{t}{a}} J_\nu \left(j_\nu\left(\frac{\log \frac{t}{a}}{\log\frac{b}{a}}\right)^{q/2} \right).
$$
\end{theorem}

The choice $\nu =1/2$ and $q=2$ in Theorem \ref{th3}, yields the following corollary. 
 \begin{corollary}
    Let $s\in\mathbb{R}$, $a$ and $b$ be any fixed numbers with $0 < a < b < \infty$. Suppose $y:[a,b]\to \mathbb{R}$ is  an absolutely continuous function such that $y(a) =y(b)= 0$ and $y'^2/t^{s-1}\in L^1[a,b]$. Then
    $$
   \left( \frac{s^2}{4}+\frac{\pi^2}{\log^2\frac{b}{a}} \right)\int\limits_a^b\frac{y^2(t)}{t^{s+1}}dt\leq\int\limits_a^b \frac{y'^2(t)}{t^{s-1}}dt.
    $$
Moreover, the equality in this inequality is attained for 
$$
Y_0(t) = t^{\frac{s}{2}}\sin\left(\pi\frac{\log \frac{t}{a}}{\log\frac{b}{a}}\right).
$$
\end{corollary}
In the next proposition we will show that the constant also is sharp in the case $\nu = 0$.

\begin{proposition}\label{prop3} For any $\varepsilon_0 >0$ there exists a function $g\in C^1[a,b]$ such that $g(a) = g(b) =0$ and $g'\in L^2[a,b] $ and 
\begin{multline*}
    \frac{s^2}{4}\int\limits_a^b\frac{g^2(t)}{t^{s+1}}dt+\frac{q^2j_0^2}{4\log^2\frac{b}{a}} \int\limits_a^b\frac{g^2(t)}{t^{s+1} \log^{2-q}\frac{t}{a}}dt+ \\ \left(\frac{1}{4}+\varepsilon_0\right)\int\limits_a^b\frac{g^2(t)}{t^{s+1}\log^2\frac{t}{a}}dt >\int\limits_a^b \frac{g'^2(t)}{t^{s-1}}dt.
\end{multline*}
\end{proposition}
\textbf{\textit{Proof of Proposition \ref{prop3}.}} To prove the statement  we apply the proof of Proposition \ref{prop2} to the function 
$$
g_\varepsilon(t) = \sqrt{t^s}\left(\log\frac{t}{a}\right)^\frac{\varepsilon+1}{2} J_0\left(j_\nu \left(\frac{\log \frac{t}{a}}{\log \frac{b}{a}}\right)^{q/2} \right).
$$ 

\section{An $L_p$-version}

In this section, $L_p$-inequalities for $p\geq 2$ are considered.    The main result of this part is the following theorem. 

\begin{theorem}\label{th4}
Suppose $p\geq 2$, $q>0$, $s>0$, $\nu\in[0,\frac{1}{q}]$, $a$ and $b$ are any fixed numbers with $0 < a < b < \infty$. Let $y:[a,b]\to \mathbb{R}$ be  an absolutely continuous function such that $y(a) = 0$ and $y'/t^{(s+1)/p}\in L^p[a,b]$. Then
       $$
\left(\frac{p}{s}\right)^p\int\limits_a^b\frac{|y'(t)|^{p}}{t^{s+1-p}}dt\geq \int\limits_a^b\left(1+\frac{pq^2\lambda^2}{2s^2\log^{2-q}\frac{t}{a}}+\frac{p(1-\nu^2q^2)}{2s^2\log^2\frac{t}{a}}\right) \frac{|y(t)|^p}{t^{s+1}}dt,
$$
where $\lambda$ is the first root of the following Lamb type equation
$$
\left(s+\frac{1}{\log\frac{b}{a}}\right)J_\nu \left(\lambda\left(\log \frac{b}{a}\right)^{q/2} \right)+q\lambda \left(\log \frac{b}{a} \right)^{q/2-1}J'_{\nu} \left(\lambda\left(\log \frac{b}{a}\right)^{q/2} \right)=0.
$$
\end{theorem}
\textbf{\textit{Proof of Theorem \ref{th4}.}} Without loss of generality, we can assume that $y$ is a positive and non-decreasing decrease. Indeed, if $g$ is an arbitrary absolutely continuous function such that $g(a)=0$ and
$$
y(t)=\int_0^t |g'(\tau)|d\tau
$$
and
\begin{equation}\nonumber
\int\limits_a^b y^p(t)w(t)dt\leq C_1\int\limits_a^b y'^p(t)v(t)dt,
\end{equation}
where  $C_1$ is a constant  and  $w$, $v$ are weight functions, then since 
$$
|g(t)|\leq\int_0^t|g'(\tau)|dt=y(t), \quad y'(t)=|g'(t)|,
$$
we have the inequality for any arbitrary absolutely continuous function
$$
\int\limits_a^b|g(t)|^pw(t)dt\leq \int\limits_a^b y^p(t) w(t)dt\leq C_1\int\limits_a^b y'^p(t)v(t)dt=C_1\int\limits_a^b |g'(t)|^pv(t)dt.
$$

Consider the function 
$$
Y_0(t) = t^{\frac{s}{2}}\sqrt{\log\frac{t}{a}} J_\nu \left(\lambda\left(\log \frac{t}{a}\right)^{q/2} \right), \quad \lambda\log \left(\log \frac{b}{a}\right)^{q/2} \in (0,j_\nu ).
$$  
Note that we use the same function as in the previous section. Using integrating by parts we have
$$
0\leq \int\limits_a^b\frac{y^{p-2}(t)}{t^{s-1}}\left(y'(t)-\frac{2}{p}\frac{Y_0'(t)}{Y_0(t)}y(t)\right)^2dt=
$$
$$
= \int\limits_a^b\frac{y^{p-2}(t) y'^2(t)}{t^{s-1}}dt-\frac{4}{p^2}\int\limits_a^b  \frac{1}{t^{s-1}}\frac{Y_0'(t)}{Y_0(t)} d y^p(t)+\frac{4}{p^2}\int\limits_a^b  \frac{ y^p(t)}{t^{s-1}}\left(\frac{Y_0'(t)}{Y_0(t)} \right)^2dt=
$$
$$
\int\limits_a^b\frac{y^{p-2}(t) y'^2(t)}{t^{s-1}}dt+\frac{4}{p^2}\int\limits_a^b\left(\frac{t^2Y_0''
(t)+(1-s)tY_0'(t)}{Y_0(t)}\right)\frac{y^p(t)}{t^{s+1}}dt
+
$$
$$
\frac{4}{p^2}\lim_{t\to a+}\frac{y^p(t)}{t^{s-1}}\frac{Y'_0(t)}{Y_0(t)}-\frac{4}{p^2}\lim_{t\to b-}\frac{y^p(t)}{t^{s-1}}\frac{Y'_0(t)}{Y_0(t)}.
$$
Hence 
$$
\int\limits_a^b\frac{y^{p-2}(t) y'^2(t)}{t^{s-1}}dt+\frac{4}{p^2}\int\limits_a^b\left(\frac{t^2Y_0''
(t)+(1-s)tY_0'(t)}{Y_0(t)}\right)\frac{y^2(t)}{t^{s+1}}dt
$$
$$
\geq
\frac{4}{p^2}\lim_{t\to b-}\frac{Y'_0(t)}{Y_0(t)}\frac{y^p(t)}{t^{s-1}}-\frac{4}{p^2}\lim_{t\to a+}\frac{Y'_0(t)}{y_0(t)}\frac{y^p(t)}{t^{s-1}}.
$$
As a consequence of the condition $y'/t^{\frac{s+1-p}{p}}\in L^p[a,b]$ via  H\"{o}lder's inequality 
$$
y^p(t)\leq \left(\int\limits_a^t|y'(\tau)|d\tau\right)^p\leq \left(\int\limits_a^t \tau^{\frac{s-p+1}{p-1}}\right)\int\limits_a^t\frac {|y'(\tau)|^p}{t^{s+1-p}}d\tau
$$
$$
\leq \left(\frac{p-1}{s}\right)^{p-1}\left(t^{\frac{s}{p-1}}-a^{\frac{s}{p-1}}\right)^{p-1}\int\limits_a^t\frac {|y'(\tau)|^p}{t^{s+1-p}}d\tau,
$$
we obtain $y^p(t)/\left(t^{\frac{s}{p-1}}-a^{\frac{s}{p-1}}\right)^{p-1}\to 0$ as $t\to a+$.

Taking into account  the following asymptotic forms
$$
\frac{Y'_0(t)}{Y_0(t) } \sim \frac{1}{2t}\left(s+\frac{1}{\log\frac{t}{a}}+   \frac{\nu q }{\log \frac{t}{a}}\right).
$$
for  small arguments $z$, we have
$$
\lim_{t\to a}\left(t^{\frac{s}{p-1}}-a^{\frac{s}{p-1}}\right)^{p-1}\frac{Y'_0(t)}{Y_0(t) } = s \frac{1+\nu q}{2a} \lim_{t\to a} t^{\frac{s}{p-1}} ( t^{\frac{s}{ p-1}} -a^{\frac{s}{p-1}})^{p-2}
$$
and
$$
\lim_{t\to a+}\frac{Y'_0(t)}{Y_0(t)}\frac{y^p(t)}{t^{s-1}}=0.
$$
Therefore
$$
\int\limits_a^b\frac{y^{p-2}(t)y'^{2}(t)}{t^{s-1}}dt+\int\limits_a^b\left(-\frac{s^2}{p^2}-\frac{q^2\lambda^2}{p^2\log^{2-q}\frac{t}{a}}-\frac{1-\nu^2q^2}{p^2\log^2\frac{t}{a}}\right) \frac{y^p(t)}{t^{s+1}}dt\geq 0
$$
and
$$
\frac{p^2}{s^2}\int\limits_a^b\frac{y^{p-2}(t)y'^{2}(t)}{t^{s-1}}dt\geq \int\limits_a^b\left(1+\frac{q^2\lambda^2}{s^2\log^{2-q}\frac{t}{a}}+\frac{1-\nu^2q^2}{s^2\log^2\frac{t}{a}}\right) \frac{y^p(t)}{t^{s+1}}dt\geq 0.
$$

Using the inequality from \cite{HLP} 
$$
a^{p_1}b^{p_2}\leq \left(\frac{p_1a+p_2b}{p_1+p_2}\right)^{p_1+p_2},
$$
for 
$$
a= \frac{y^p(t)}{t^s}, \quad b =\frac{p^p}{s^p} \frac{y'^p(t)}{t^{s+1-p}},\quad p_1 = 1- \frac{2}{p} \quad \text{ and  }\quad p_2 = \frac{2}{p},
$$
we have 
$$
\frac{p^p}{s^p}\int\limits_a^b\frac{y'^{p}(t)}{t^{s-1}}dt\geq \int\limits_a^b\left(1+\frac{pq^2\lambda^2}{2s^2\log^{2-q}\frac{t}{a}}+\frac{p(1-\nu^2q^2)}{2s^2\log^2\frac{t}{a}}\right) \frac{y^p(t)}{t^{s+1}}dt.
$$
By equality (\ref{f_difeq}) we  complete the proof of Theorem \ref{th4}.

The application of the theorem  with $\nu =1/2$, $q= 2$ and $s=p-1$ gives
\begin{corollary}
Let $p\geq 2$, $a$ and $b$ be any fixed numbers with $0 < a < b < \infty$. Suppose $y:[a,b]\to \mathbb{R}$ is an absolutely continuous function such that $y(a) = 0$ and $y'\in L^p[a,b]$. Then
    $$
\left(\frac{p}{p-1}\right)^p\left(1+\frac{2p\lambda^2}{(p-1)^2}\right)^{-1}\int\limits_a^b\frac{|y'(t)|^{p}}{t^{p-2}}dt\geq \int\limits_a^b \frac{|y(t)|^p}{t^{p}}dt,
$$
where $\lambda$ is the first root of the following Lamb type equation
$$
2\lambda \cos\left(\lambda\log \frac{b}{a} \right)+s\sin\left(\lambda \log\frac{b}{a}\right) =0.
$$
\end{corollary}

Now, if we consider continuously differentiable functions such that $y(a) =y(b)= 0$ and put that 
$$
Y_0(t) = t^{\frac{s}{2}}\sqrt{\log\frac{t}{a}} J_\nu \left(j_\nu\left(\frac{\log \frac{t}{a}}{\log\frac{b}{a}}\right)^{q/2} \right),
$$
then using H\"{o}lder's inequality, we have 
$$
y^p(b)\leq \left(\int\limits_t^b|y'(\tau)|d\tau\right)^p\leq \left(\int\limits_t^b \tau^{\frac{s-p+1}{p-1}}\right)\int\limits_t^b\frac {|y'(\tau)|^p}{t^{s+1-p}}d\tau
$$
$$
\leq \left(\frac{p-1}{s}\right)^{p-1}\left(b^{\frac{s}{p-1}}-t^{\frac{s}{p-1}}\right)^{p-1}\int\limits_t^b\frac {|y'(\tau)|^p}{t^{s+1-p}}d\tau.
$$
Hence $y^p(t)/\left(b^{\frac{s}{p-1}} - t^{\frac{s}{p-1}}\right)^{p-1}\to 0$ as $t\to b-$. Combining these facts with the proof of Theorem \ref{th4} we have the following theorem.

\begin{theorem}
Suppose $p\geq 2$, $q>0$, $\nu\in[0,\frac{1}{q}]$, $a$ and $b$ are any fixed numbers with $0 < a < b < \infty$. Let $y:[a,b]\to \mathbb{R}$ be  an absolutely continuous function such that $y(a) = y(b)= 0$ and $y'/t^{(s+1)/p}\in L^p[a,b]$. Then
       $$
\left(\frac{p}{s}\right)^p\int\limits_a^b\frac{|y'(t)|^{p}}{t^{s+1-p}}dt\geq \int\limits_a^b\left(1+\frac{pq^2j_\nu^2}{2s^2\log^2\frac{b}{a}\log^{2-q}\frac{t}{a}}+\frac{p(1-\nu^2q^2)}{2s^2\log^2\frac{t}{a}}\right) \frac{|y(t)|^p}{t^{s+1}}dt.
$$
\end{theorem}
\section{Inequalities in an annuli}

In this section, using one-dimensional inequalities, we obtain spatial Hardy type inequalities. Inequalities in annuli are established.  Suppose  $n\in \mathbb{N}$, $0<r_1<r_2<\infty$  and $$A_n(r_1,r_2) =\{x\in\mathbb{R}^n: r_1<|x|< r_2\}$$
is  the open multidimensional annulus.

The following theorem holds. 

\begin{theorem}  \label{th6} Suppose $n\in \mathbb{N}$, $0<r_1<r_2<\infty$,  $f\in C_0^\infty(A_n(r_1,r_2))$,  $q>0$, $s>0$ and $\nu\in[0,\frac{1}{q}]$. Then the following sharp inequality holds
\begin{multline}\label{th5_in}
    \frac{s^2}{4}\int\limits_{A_n(r_1,r_2)}\frac{f^2(x)}{|x|^{s+n}}dx+\frac{q^2j_\nu^2}{4\log^2\frac{r_1}{r_2}} \int\limits_{A_n(r_1,r_2)}\frac{f^2(x)}{|x|^{s+n} \log^{2-q}\frac{|x|}{r_1}}dx+\\ \frac{1-\nu^2q^2}{4}\int\limits_{A_n(r_1,r_2)}\frac{f^2(x)}{|x|^{s+n}\log^2\frac{|x|}{r_1}}dx \leq\int\limits_{A_n(r_1,r_2)}\frac{|\nabla f(x)|^2}{|x|^{s+n-2}}dx.
\end{multline}
Moreover, if $\nu\in(0,\frac{1}{q}]$ then the equality in this inequality is attained for 
$$
Y_0(r) = r^{\frac{s}{2}}\sqrt{\log\frac{r}{r_1}} J_\nu \left(j_\nu\left(\frac{\log \frac{t}{r_1}}{\log\frac{r_2}{r_1}}\right)^{q/2} \right).
$$
\end{theorem}

\textbf{\textit{Proof of Theorem \ref{th6}.}} We denote by $\mathbb{S}^{n-1}$ the $(n-1)$-dimensional unit sphere in $\mathbb{R}^n$, by $d^{n-1}\omega$ the usual volume measure on $\mathbb{S}^{n-1}$, by $\nabla f$ the usual gradient operator. Using Theorem \ref{th3}, we obtain 
\begin{multline*}    
    \frac{s^2}{4}\int\limits_{r_1}^{r_2}\frac{f^2(r,\theta)}{r^{s+n}}r^{n-1}dr+\frac{q^2j_\nu^2}{4\log^2\frac{r_1}{r_2}} \int\limits_{r_1}^{r_2}\frac{f^2(r,\theta)}{t^{s+n} \log^{2-q}\frac{r}{r_1}}r^{n-1}dr+\\ \frac{1-\nu^2q^2}{4}\int\limits_{r_1}^{r_2}\frac{f^2(r,\theta)}{r^{s+n}\log^2\frac{r}{r_1}}r^{n-1}dr
    \leq\int\limits_{r_1}^{r_2}\left(\frac{\partial f(r,\theta)}{\partial r}\right)^2\frac{r^{n-2}}{r^{s+n-2}}dr,
\end{multline*}
where  $r\in(0,\infty)$ and $\theta\in \mathbb{S}^{n-1}$.

Integrating both sides with respect to the normalized surface
measure on $\mathbb{S}^{n-1}$, we have
\begin{multline*} 
    \frac{s^2}{4}\int\limits_{A_n(r_1,r_2)}\frac{y^2(x)}{|x|^{s+n}}dx+\frac{q^2j_\nu^2}{4\log^2\frac{r_1}{r_2}} \int\limits_{A_n(r_1,r_2)}\frac{f^2(x)}{|x|^{s+n} \log^{2-q}\frac{|x|}{r_1}}dx+\\ 
    \frac{1-\nu^2q^2}{4}\int\limits_{A_n(r_1,r_2)}\frac{f^2(x)}{|x|^{s+n}\log^2\frac{|x|}{r_1}}dx
    \leq\int\limits_{A_n(r_1,r_2)}\left(\frac{\partial f(x)}{\partial r}\right)^2\frac{dx}{|x|^{s+n-2}}.
\end{multline*}
Taking into account the inequality
$$
\left|\frac{\partial f}{\partial r}\right|^2 \leq |\nabla f|^2 = \left|\frac{\partial f}{\partial r}\right|^2+  \frac{1}{r^2}|\nabla_{\theta}f|^2,
$$
where $x = (r, \theta)\in\mathbb{R}^n$, $r\in(0,\infty)$ and $\theta\in \mathbb{S}^{n-1}$,
we get inequality (\ref{th5_in}). 

To prove that the constants are sharp, we use radial functions of the form $f(r,\theta) = v(r)$, where $v\in C_0^1(r_1,r_2)$. For such functions the constants being sharp in (\ref{th5_in}) is easily seen to be equivalent to the constants in (\ref{th3_ineq}) being sharp over the set of functions $v\in C_0^1(r_1,r_2)$. This is indeed so because the functions which were employed in Theorem \ref{th3} and Proposition \ref{prop3} to prove the sharpness of the constants.

\section*{Acknowledgement}
The research is supported by a grant of Russian Science Foundation (project no. 23-11-00066-П,  https://rscf.ru/en/project/23-11-00066-П)

\section*{Declarations}
The author report there are no competing interests to declare.

\end{document}